\documentclass[10pt]{amsart}
\usepackage{mathrsfs}
\usepackage{enumerate, float}
\usepackage[colorlinks = true,
linkcolor = blue,
urlcolor  = blue,
citecolor = blue,
anchorcolor = blue]{hyperref}

\usepackage{amsmath}
\usepackage{amssymb, indentfirst}
\usepackage{tikz}
\usetikzlibrary{matrix,arrows}
\tikzset{global scale/.style={
		scale=#1,
		every node/.style={scale=#1}
	}
}
\usetikzlibrary{patterns}

\def\PP{\mathbb{P}}

\def\RR{\mathbb{R}}

\def\ZZ{\mathbb{Z}}

\def\shF{\mathscr{F}}

\def\shO{\mathcal{O}}

\newcommand\conv{conv}
\newcommand\Pic{Pic}

\newcommand\vol{vol}
\newcommand\Vol{Vol}

\newcommand\divisor{div}

\newcommand{\vt}{\hspace{1mm}\middle\vert\hspace{1mm}}
\newtheorem{theorem}{Theorem}[section]
\newtheorem{corollary}[theorem]{Corollary}
\newtheorem{proposition}[theorem]{Proposition}
\newtheorem{lemma}[theorem]{Lemma}
\newtheorem{conjecture}[theorem]{Conjecture}

\theoremstyle{remark}
\newtheorem{remark}[theorem]{Remark}
\newtheorem{example}[theorem]{Example}
\theoremstyle{definition}

\begin{document}
	\title{A Reider-type Result for Smooth Projective Toric Surfaces}
	\author{Bach Le Tran}
	\address{School of Mathematics\\ University of Edinburgh\\ James Clerk Maxwell Building\\ Peter Guthrie Tait Road\\ Edinburgh EH9 3FD}
	\begin{abstract}
		Let $L$ be an ample line bundle over a smooth projective toric surface $X$. Then $L$ corresponds to a very ample lattice polytope $P$ that encodes many geometric properties of $L$. In this article, by studying $P$, we will give some necessary and sufficient numerical criteria for the adjoint series $|K_X+L|$ to be either nef or (very) ample.
	\end{abstract}
	\maketitle

\section{Introduction}

The problem of determining whether a line bundle is nef or (very) ample is an important question in algebraic geometry. The Nakai-Moishezon criterion \cite{Nakai1963, Moishezon1964} states that a Cartier divisor $D$ on a proper scheme $X$ over an algebraically closed field is ample if and only if $D^{\dim(Y)}\cdot Y > 0$ for every closed integral subscheme $Y$ of $X$. For toric varieties, a special form of the criterion holds: if $D\cdot C>0$ for every torus-invariant curve $C\subset X$ then $D$ is ample. Furthermore, if $D\cdot C\ge 0$ for every torus-invariant curve $C\subset X$ then $D$ is globally generated \cite{Laterveer1996, Mavlyutov2000, Mustata2002}. However, the question is more complicated when we consider the adjoint bundle $D+K_X$. Namely, are there numerical conditions for $D\cdot C$ so that $D+K_X$ is globally generated or ample? Fujita conjectured the following:

\begin{conjecture}[\cite{Fujita1985}]\label{Fujita conjecture}
	Let $X$ be an $n$-dimensional projective algebraic variety, smooth or with mild singularities, and $D$ an ample divisor on $X$. Then
	\begin{enumerate}[(1)]
		\item For $t\ge n+1$, $tD+K_X$ is basepoint free.
		\item For $t\ge n+2$, $tD+K_X$ is very ample.
	\end{enumerate}
\end{conjecture}

The conjecture is true for toric varieties \cite{Fujino2003, Payne2006}. For smooth surfaces, Fujita's conjecture follows from Reider's theorem \cite{Reider1988}.

In this article, we will present a combinatorial proof for a Reider-type result for smooth projective toric surfaces.

\begin{proposition}\label{My Reider variant}
	Let $X$ be a smooth projective toric surface not isomorphic to $\PP^2$, and let $L$ be an ample line bundle on $X$.
	\begin{enumerate}
		\item The adjoint series $|K_X+L|$ is not base point free if and only if there exists an effective torus-invariant divisor $D\subset X$ such that 
		\begin{align*}
		D\cdot L=1&\text{ and } D^2=0.
		\end{align*}
		
		\item The adjoint series $|K_X+L|$ is not ample if and only if there exists an effective torus-invariant divisor $D\subset X$ such that either
		\begin{align*}
		D\cdot L=1 &\text { and } D^2=-1 \text{ or } D^2=0 \text{; or}\\
		D\cdot L=2&\text{ and } D^2=0\text{; or}\\
		D\cdot L=3&\text{ and } D^2=1.
		\end{align*}
		Furthermore, if $L^2\ge 10$, then $|K_X+L|$ is not ample if and only if there exists an effective torus-invariant divisor $D\subset X$ such that either 
		\begin{align*}
		D\cdot L=1 &\text { and } D^2=-1 \text{ or } D^2=0 \text{; or}\\
		D\cdot L=2&\text{ and } D^2=0.
		\end{align*}
	\end{enumerate} 
\end{proposition}

As a convention, in this article, we will follow the notations in \cite{Cox2011}. In particular, we will always use $M$ to denote the ambient lattice if there is no confusions.

\subsection*{Acknowledgments}
We would like to thank Milena Hering for suggesting the problem and for her invaluable guidance. We also want to thank Ivan Cheltsov for some of the comments.
\section{Toric Surfaces Reviewed}
Let $A$ be an ample line bundle over a projective toric variety $X$ corresponding to a polytope $P\subset M_{\RR}$. Then we have a combinatorial interpretation of the intersection number $A\cdot C$ where $C\subset X$ is any torus-invariant curve as follows.
\begin{lemma}[{\cite[(1.4) and Page 457]{Laterveer1996}}]\label{lattice length}
	Let $A$ be an ample line bundle on a projective toric variety $X$ corresponding to a polytope $P$. For a torus invariant curve $C$, let $E$ be the corresponding edge on $P$. Then $A\cdot C$ is equal to the lattice length of $E$, i.e., 
	\[A\cdot C=|E\cap M|-1.\]
\end{lemma}

For our purpose, we will need to use the classification of smooth projective toric surfaces: every smooth complete toric surfaces is a finite blowup of either $\PP^2$, $\PP^1\times\PP^1$, or the Hirzebruch surface $\shF_a$, where $a\ge 2$ ( {\cite[Theorem 10.4.3]{Cox2011}}). Another important fact that we will use is that every ample line bundle on a smooth projective toric surface is also very ample.

\begin{lemma}[{\cite[Theorem 6.1.15]{Cox2011}}]
	A line bundle on a smooth complete toric variety is ample if and only if it is very ample. 
\end{lemma}

Smooth toric surfaces are interesting objects to work with; partially because of their computability. For example, we have the following lemma.
\begin{lemma}[{\cite[Proposition 10.4.11]{Cox2011}}]\label{K_X x D_i=b_i-2}
	Let $u_0,\ldots, u_r$ be ray generators of a smooth complete fan $\Sigma$ in $N_{\RR}\cong \RR^2$. Let $X=X_{\Sigma}$ be the smooth projective toric surface from $\Sigma$ and $D_i=V(u_i)$ for $0\le i\le r$. Let $K_X$ be the canonical divisor $K_X=-\sum_{i=0}^rD_i$. Then
	\[K_X\cdot D_i=b_i-2,\]
	where the $b_1,\ldots,b_{r-1}$ are integers such that $u_{i-1}+u_{i+1}=b_iu_i$ for all $0\le i\le r$, where $u_{-1}=u_r$ and $u_{r+1}=u_0$.
\end{lemma}

The following corollary follows directly from \cite[Lemma 10.4.1]{Cox2011} and Lemma \ref{K_X x D_i=b_i-2}.
\begin{corollary}\label{(L+KX)D=LD-D^2-2}
	Let $u_0,\ldots, u_r$ be ray generators of a smooth complete fan $\Sigma$ in $N_{\RR}\cong \RR^2$. Let $X=X_{\Sigma}$ be the smooth projective toric surface from $\Sigma$ and $D_i=V(u_i)$ for $0\le i\le r$. Let $K_X$ be the canonical divisor $K_X=-\sum_{i=0}^rD_i$. Then for $0\le i\le r$,
	\[(L+K_X)\cdot D_i=L\cdot D_i-D_i^2-2.\]
\end{corollary}

We also know that the blowup of a toric variety corresponds to a subdivision of fan. Thus the number of generating rays of the fan corresponding to a toric surface increases after a blowup (\cite[Proposition 3.3.15]{Cox2011}). 

\begin{example}\label{Nef cone if Hirzebruch surface}
	Consider the Hirzebruch surface $\shF_r=\PP(\shO_{\PP^1}\oplus\shO_{\PP^1}(r))$, $r\ge 1$, whose fan $\Sigma$ given by the following figure
	\begin{figure}[H]\label{Hirzebruch fan}
		\begin{center}
			\begin{tikzpicture}[scale=1]		
			\coordinate  (5) at (0,0);
			\coordinate  (6) at (2,0);
			\coordinate  (7) at (0,2);
			\coordinate  (8) at (-1,2);
			\coordinate  (9) at (0,-2);
			
			\fill[pattern=dots, opacity=.5] (6)--(7)--(5)--(6);
			\fill[pattern=checkerboard light gray, fill opacity=0.5] (5)--(8)--(9)--(5);
			\fill[pattern=crosshatch dots gray, fill opacity=.5] (5)--(9)--(6)--(5);
			\fill[fill=lightgray, fill opacity=.5] (5)--(7)--(8)--(5);
			\node at (.6,.6) {$\sigma_1$};
			\node at (.6,-.6) {$\sigma_2$};
			\node at (-.25,0) {$\sigma_3$};
			\node at (-.3,1.5) {$\sigma_4$};
			\node [left] at (-.5,1) {$(-1,r)$};
			\draw[thick,->] (5)--(6);
			\draw[thick,->] (5)--(7);
			\draw[thick,->] (5)--(8);
			\draw[thick,->] (5)--(9);
			\filldraw (0,0) circle (1pt); 
			\filldraw (-.5,1) circle (1pt);
			\end{tikzpicture}
			\caption{The Hirzebruch fan} 
		\end{center}
	\end{figure}
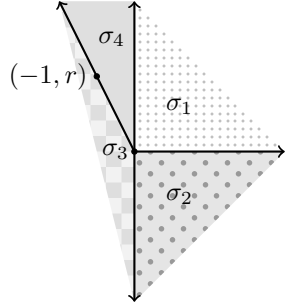
	The ray generators of $\Sigma$ are $v_1=(1,0)$, $v_2=(0,1)$, $v_3=(-1,r)$, and $v_4=(0,-1)$. Let the associated divisors be $D_1$, $D_2$, $D_3$, and $D_4$, respectively. By \cite[Proposition 4.1.2]{Cox2011}, 
	\begin{align*}
	0\sim \divisor(\chi^{e_1})&=\sum_{i=1}^4\langle e_1,v_i\rangle D_i= D_1-D_3\\
	0\sim \divisor(\chi^{e_2})&=\sum_{i=1}^4\langle e_2,v_i\rangle D_i=D_2+aD_3-D_4.
	\end{align*}
	Thus $D_3\sim D_1$, $D_4\sim D_2+aD_3$, and
	\[\Pic(\shF_r)\simeq \left\{aD_3+bD_4\vt a,b\in\ZZ\right\}.\]
	The maximal cones of $\Sigma$ are $\sigma_1$, $\sigma_2$, $\sigma_3$ and $\sigma_4$ as in Figure \ref{Hirzebruch fan}. Let $D=aD_3+bD_4$. We compute the $m_{\sigma_i}$ to be
	\[m_1=(-a,0),\hspace{2mm} m_2=(-a,b),\hspace{2mm} m_3=(rb,b),\hspace{2mm} m_4=(0,0).\]
	Then by \cite[Lemma 6.1.13]{Cox2011}, $D$ is very ample if and only if $a,b>0$. The nef cone of $\shF_r$ is given by
	\begin{figure}[H]
		\begin{center}	
			\begin{tikzpicture}[scale=1.5]
			\draw [thick,->](0,0) --(0,2);
			\draw [thick,->](0,0) --(2,0);
			\draw [dashed] (0,0)--(0,-1);
			\draw [dashed] (0,0)--(-1,0);
			\draw [fill] (0,0) circle [radius=0.025];
			\draw [fill] (1,0) circle [radius=0.025];
			\draw [fill] (0,1) circle [radius=0.025];
			\node [left] at (0,1) {$[D_4]$};				
			\node [below] at (1,0) {$[D_3]$};
			\fill [pattern=dots, fill opacity=0.2] (0,0)--(0,2)--(2,2)--(2,0)--(0,0);
			\end{tikzpicture}	
			\caption{The nef cone of $\shF_r$}	
		\end{center}
	\end{figure}
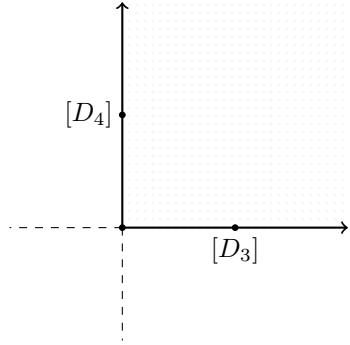
	
	By \cite[Lemma 10.4.1]{Cox2011}, we have $D_1^2=D_3^2=0$, $D_2^2=-a$, $D_4^2=a$, $D_1\cdot D_2=D_2\cdot D_3=D_3\cdot D_4=D_4\cdot D_1=1$, $D_1\cdot D_3=D_2\cdot D_4=0$. 
	
\end{example}

Finally, we will make use of the Hodge's Index Theorem:
\begin{lemma}[{\cite[Theorem V.1.9]{Hartshorne1977}}]\label{Hodge index theorem}
	Let $D$ be an ample divisor on a smooth projective surface $S$. If $E$ is a divisor such that $D\cdot E=0$, then $E^2\le 0$. The equality occurs if and only if $E$ is numerically equivalent to $0$.
\end{lemma}

\begin{corollary}[{\cite[Exercise V.1.9]{Hartshorne1977}}]\label{Hodge inequality}
	Let $D$ be an ample divisor on a smooth projective surface $S$ and $E$ an arbitrary divisor. Then 
	\[(D\cdot E)^2\ge D^2 E^2.\]
\end{corollary}

\begin{proof}
	Since $D$ is ample, $D^2>0$. Let $H=(D^2)E-(D\cdot E)D$. We have
	\[D\cdot H=(D^2)E\cdot D-(D\cdot E)D^2=0.\]
	Then by Lemma \ref{Hodge index theorem}, we must have $H^2\le 0$. In other words,
	\begin{align*}
		0\ge &\left((D^2)E-(D\cdot E)D\right)\cdot \left((D^2)E-(D\cdot E)D\right)\\
		= &D^4E^2-2(D\cdot E)^2(D^2)+D^2(D\cdot E)^2\\
		= &D^2\left(D^2E^2-(D\cdot E)^2\right).
	\end{align*}
	Since $D^2>0$, it follows that $(D\cdot E)^2\ge D^2E^2$.
\end{proof}

\section{Toric Surfaces and Lattice Polygons}
In this section, we review and prove some lemmas on lattice polygons that we will use to the proof of Proposition \ref{My Reider variant}.
\begin{lemma}[{\cite[Lemma 1]{Arkinstall1980}}]\label{pigeonhole}
	Every lattice polygon with at least 5 edges has at least an interior lattice point. 
\end{lemma}

\begin{lemma}\label{second pigeonhole}
	Let $v_1,\ldots,v_5$ be lattice points such that no three points are collinear. Then there exists a lattice point in $\conv(v_1,\ldots,v_5)\backslash\{v_1,\ldots,v_5\}$.
\end{lemma}
\begin{proof}
	Let the coordinates of $v_i$ be $(x_i,y_i)$ for $i=1,\ldots, 5$. By the pigeonhole principle, there must be $i\neq j$ such that $x_i\equiv x_j\pmod 2$ and $y_i\equiv y_j\pmod 2$. Then the midpoint $m$ of $v_iv_j$ is a lattice point. Since no three points in $\{v_1,\ldots,v_5\}$ are collinear, it follows that $m\in \conv(v_1,\ldots,v_5)\backslash\{v_1,\ldots,v_5\}$.
\end{proof}

As a consequence, we obtain:
\begin{lemma}\label{Volume of polygon >=9}
	Let $P$ be a lattice polygon that has at least $5$ vertices and assume that one of its edges has lattice length $4$. Then $\Vol(P)\ge 9$.
\end{lemma}
\begin{proof}
	It suffices to prove the lemma when $P$ is a lattice pentagon. Let $P=\conv(v_1,\ldots,v_5)$, where $v_1,\ldots,v_5$ are ordered clockwise in $M$. Without loss of generality suppose that the lattice length of the edge joining $v_1$ and $v_5$ is $4$; i.e., there are $3$ other lattice points $y_1$, $y_2$, $y_3$ in between $v_1$ and $v_5$.
	
	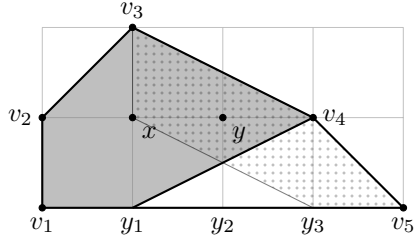
\begin{figure}[H]\label{pentagon}
		\begin{center}
			\begin{tikzpicture}[scale=1.2]		
			\draw [help lines, opacity=.5] (0,0) grid (4,2);
			\coordinate  (1) at (0,0);
			\coordinate  (2) at (0,1);
			\coordinate  (3) at (1,2);
			\coordinate  (4) at (3,1);
			\coordinate  (5) at (4,0);
			\coordinate  (6) at (1,1);
			\coordinate  (7) at (1,0);
			\coordinate  (8) at (2,0);
			\coordinate  (9) at (3,0);
			\coordinate  (10) at (2,1);	
			\fill [color=gray, opacity=.5] (1)--(2)--(3)--(4)--(7)--(1);
			\filldraw [pattern=dots, opacity=.5] (3)--(6)--(9)--(5)--(4)--(3);
			\draw [thick] (1)--(2)--(3)--(4)--(5)--(1);
			\draw [thick] (4)--(7);
			
			\node [below] at (1) {$v_1$};
			\node [left] at (2) {$v_2$};
			\node [above] at (3) {$v_3$};
			\node [right] at (4) {$v_4$};
			\node [below] at (5) {$v_5$};
			\node [below right] at (6) {$x$};
			\node [below right] at (10) {$y$};
			\node [below] at (7) {$y_1$};
			\node [below] at (8) {$y_2$};
			\node [below] at (9) {$y_3$};
			\filldraw (1) circle (1pt); 
			\filldraw (2) circle (1pt); 
			\filldraw (3) circle (1pt); 
			\filldraw (4) circle (1pt); 
			\filldraw (5) circle (1pt); 
			\filldraw (6) circle (1pt); 
			\filldraw (10) circle (1pt);
			\end{tikzpicture}
			\caption{A lattice pentagon that has an edge whose lattice length is $4$} 
		\end{center}
	\end{figure} 
	Consider the polytope $Q=\conv(v_1,v_2,v_3,v_4,y_1)$. Then by Lemma \ref{pigeonhole}, there must be a lattice point $x$ in the interior of $Q$. Then $x$ lies in at most one of the segments $v_1v_3$, $y_1v_3$, $y_2v_3$, $y_3v_3$, $v_5v_3$. If $x$ lies in $v_1v_3$ or if $x$ does not lie in any mentioned segments, consider the set of $5$ points $\{x,v_3,v_4,v_5,y_1\}$. By Lemma \ref{second pigeonhole}, there must be another lattice point $y$ in $P$ that is not the same as the points listed before. If $y\in \partial P$, then $|\partial P\cap M|\ge 9$ and $|P^0\cap M|\ge 1$. By Pick's theorem \cite{Pick1899},
	\[\Vol(P)=|\partial P\cap M|+2|P^0\cap M|-2\ge 9.\]
	If $y\in P^0$, then $|\partial P\cap M|\ge 8$ and $|P^0\cap M|\ge 2$. Again, by Pick's theorem,
	\[\Vol(P)=|\partial P\cap M|+2|P^0\cap M|-2\ge 10.\] 
	If $x$ lies in $v_3y_1$ or $v_3y_2$ then we get such a point $y$ from $\conv(x,v_3,v_4,v_5,y_3)$. If $x$ lies in $v_3y_3$ or $v_3v_5$ then we get $y$ from $\conv(v_1,v_2,v_3,x,y_2)$. The same argument follows and we proved the lemma. 
\end{proof}

We will also need the following lemmas for the proof of Proposition \ref{My Reider variant}.
\begin{lemma}\label{L^2>=LD+4}
	Let $L$ be an ample line bundle over a smooth projective toric surface $X$. Let $\Sigma$ be the fan of $X$. Suppose that $\Sigma$ has $n\ge 5$ rays $\rho_1,\ldots, \rho_n$. Then for any integer $1\le i\le n$,
	\[L^2\ge L\cdot D_{\rho_i}+4.\]
\end{lemma}
\begin{proof}
	Let $P$ be the polytope associated to $L$. By Pick's theorem (\cite{Pick1899}) and since $L$ is ample so that $L\cdot D_{\rho_i}\ge 1$ for all $i$,
	\begin{equation}\label{Pick}
		\vol(P)=\frac{L^2}{2}=\frac{|\partial P\cap M|}{2}+|P^0\cap M|-1,
	\end{equation}
	where $\partial P$ and $P^0$ are the sets of all boundary points and interior points of $P$, respectively. By Lemma \ref{lattice length},
	\begin{equation}\label{perimeter of polytope}
		|\partial P\cap M|=\sum_{j=1}^nL\cdot D_{\rho_j}.
	\end{equation}
	Hence, combining \eqref{Pick} and \eqref{perimeter of polytope} gives
	\[L^2=\sum_{j=1}^n L\cdot D_{\rho_j}+2|P^0\cap M|-2
	\ge L\cdot D_{\rho_i}+(n-1)+2|P^0\cap M|-2.\]
	Since $n\ge 5$, by Lemma \ref{pigeonhole}, $|P^0\cap M|\ge 1$. Therefore,
	\[L^2\ge L\cdot D_{\rho_i}+4.\]
\end{proof}

\section{A Reider-type Result for Toric Surfaces}
We will devote this section to prove Proposition \ref{My Reider variant}. First of all, it is true for $X\cong \PP^1\times \PP^1$.
\begin{lemma}\label{P1xP1}
	Proposition \ref{My Reider variant} holds for $X\cong\PP^1\times\PP^1$.
\end{lemma}
\begin{proof}
	Let $\Sigma$ be the fan of $X=\PP^1\times \PP^1$ as follows.
	\begin{figure}[H]
		\begin{center}
			\begin{tikzpicture}[scale=2]   
			\fill[pattern=dots, opacity=.5] (0,0)--(1,0)--(1,1)--(0,1)--(0,0);
			\fill[pattern=checkerboard light gray, fill opacity=0.5] (0,0)--(0,1)--(-1,1)--(-1,0)--(0,0);
			\fill[pattern=crosshatch dots gray, fill opacity=.5] (0,0)--(-1,0)--(-1,-1)--(0,-1)--(0,0);
			\fill[fill=lightgray, fill opacity=.5] (0,0)--(0,-1)--(1,-1)--(1,0)--(0,0);
			\draw[thick,->] (0,0)--(0,1);
			\draw[thick,->] (0,0)--(1,0);
			\draw[thick,->] (0,0)--(-1,0);
			\draw[thick,->] (0,0)--(0,-1);
			\node at (.5,.5) {$\sigma_1$};
			\node at (.5,-.5) {$\sigma_2$};
			\node at (-.5,-.5) {$\sigma_3$};
			\node at (-.5,.5) {$\sigma_4$};
			\end{tikzpicture}
			\caption{The fan of $\PP^1\times\PP^1$}
		\end{center}
	\end{figure}
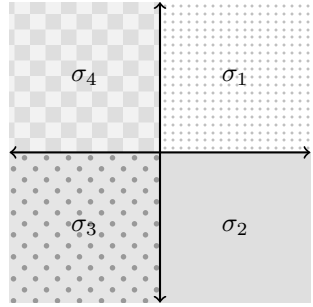
	By \cite[Lemma 10.4.1]{Cox2011}, $D_{\rho}^2=0$ for all $\rho\in\Sigma(1)$. Thus, we need to show that there exists $\rho$ such that $L\cdot D_{\rho}=1$ in the first part and $L\cdot D_{\rho}\le 2$ in the second part.
	
	For any ample bundle $L$ on $X$, if $L+K_X$ is not basepoint free, then there exists $\rho\in \Sigma(1)$ such that $(L+K_X)\cdot D_{\rho}< 0$. Then By lemma \ref{K_X x D_i=b_i-2}, 
	\[(L+K_X)\cdot D_{\rho}=L\cdot D_{\rho}-D_{\rho}^2-2<0.\]
	This implies $0<L\cdot D_{\rho}< D_{\rho}^2+2=2$, so that $L\cdot D_{\rho}=1$.
	
	Now suppose that $L+K_X$ is not ample and $(L+K_X)\cdot D_{\rho}\le 0$. Then By lemma \ref{K_X x D_i=b_i-2},
	\[(L+K_X)\cdot D_{\rho}=L\cdot D_{\rho}-D_{\rho}^2-2\le 0.\]
	This implies $1\le L\cdot D_{\rho}\le  D_{\rho}^2+2=2$. Hence, either $L\cdot D_{\rho}=1$ and $D_{\rho}^2=0$ or $L\cdot D_{\rho}=2$ and $D_{\rho}^2=0$. The conclusion follows.
\end{proof}
Secondly, we show that Proposition \ref{My Reider variant} holds for Hirzebruch surfaces.
\begin{lemma}\label{Hirzebruch}
	Proposition \ref{My Reider variant} holds for $X\cong \shF_a$, $a\ge 1$.
\end{lemma}
\begin{proof}
	Consider the Hirzebruch surface $X=\shF_r=\PP(\shO_{\PP^1}\oplus\shO_{\PP^1}(r))$, $r\ge 1$ as in Example \ref{Nef cone if Hirzebruch surface}. We have
	\[\Pic(\shF_r)\simeq \left\{aD_3+bD_4\vt a,b\in\ZZ\right\}.\]
	The canonical divisor of $X$ is given by
	\[K_X=-(D_1+D_2+D_3+D_4)\sim -(2-a)D_3-2D_4.\]
	Recall that $D_1^2=D_3^2=0$, $D_2^2=-a$, $D_4^2=a$, $D_1\cdot D_2=D_2\cdot D_3=D_3\cdot D_4=D_4\cdot D_1=1$, and $D_1\cdot D_3=D_2\cdot D_4=0$ (cf. \cite[Lemma 10.4.1]{Cox2011}). 
	
	Let $L$ be an ample line bundle over $\shF_r$. Then $L^2>0$. We have two cases as follows.
	\begin{itemize}
		\item If $r=1$ then $K_X=-D_3-2D_4$. For $L$ to be ample while $L+K_X$ is not nef, $L$ has to be of the form $L\sim cD_3+D_4$, $c>0$. In this case, take $D=D_3$, then
		\[L\cdot D=1\text{ and } D^2=0.\]
	
		For $L$ to be ample while $L+K_X$ is not ample, $L$ has to be of the form $L\sim D_3+cD_4$, or $L\sim cD_3+D_4$, or $L\sim cD_3+2D_4$, where $c\ge 1$.
	
		\begin{enumerate}
			\item If $L\sim D_3+cD_4$, take $D=D_2$, then
			\[L\cdot D=1\text{ and } D^2=-1.\]
			\item If $L\sim cD_3+D_4$, take $D=D_3$, then
			\[L\cdot D=1\text{ and } D^2=0.\]
			\item If $L\sim cD_3+2D_4$, take $D=D_3$, then
			\[L\cdot D=2\text{ and } D^2=0.\]
		\end{enumerate}	
		
		\item $r\ge 2$:
		For $L$ to be ample but $K_X+L$ is not nef, $L$ has the form
		\[L\sim D_4+cD_3 \hspace{5mm} (c\ge 0).\]
		Take $D=D_3$, then $L\cdot D=1\text{ and } D^2=0$.
		
		For $L$ to be ample but $K_X+L$ is not, $L$ has the form $L\sim cD_3+D_4$ or $L\sim cD_3+2D_4$, where $c\ge 1$.
		
		\begin{enumerate}
			\item If $L\sim cD_3+D_4$, take $D=D_3$, then
			\[L\cdot D=1\text{ and } D^2=0.\]
			\item If $L\sim cD_3+2D_4$, take $D=D_3$, then
			\[L\cdot D=2\text{ and } D^2=0.\]
		\end{enumerate}
		
	\end{itemize}
\end{proof}

Finally, we will give the proof for the final case of Proposition \ref{My Reider variant}, when $X$ is an arbitrary blowup of $\PP^1\times\PP^1$ or the Hirzebruch surface. 
\begin{proof}[Proof of Proposition \ref{My Reider variant}]\label{proof of Reider}

	The sufficiency trivially holds by Corollary \ref{(L+KX)D=LD-D^2-2}. We now prove the necessity.

	By the classification of smooth projective toric surfaces, the proofs for the cases of $\PP^1\times\PP^1$ (Lemma \ref{P1xP1}) and $\shF_a$ (Lemma \ref{Hirzebruch}), it suffices to prove the proposition in the case that the fan $\Sigma$ of $X$ has at least $5$ rays.
	
	We first prove part 1. Suppose that $K_X+L$ is not basepoint free. Then there exists $\rho\in \Sigma(1)$ such that $(K_X+L)\cdot D_{\rho}<0$. Take $D=D_{\rho}$. By Lemma \ref{K_X x D_i=b_i-2},
	\[(L+K_X)\cdot D=L\cdot D-D^2-2<0.\]
	This implies $L\cdot D< D^2+2$, so since $L$ is ample,
	\begin{equation}\label{eq1}
		0\le L\cdot D-1\le D^2.
	\end{equation}
	\begin{itemize}
		\item If $D^2\le -1$, then $L\cdot D\le 0$, which is a contradiction to the hypothesis that $L$ is ample.
		\item If $D^2=0$, either $D\cdot L=0$ or $D\cdot L=1$. But $D\cdot L>0$ since $L$ is ample. Thus $D\cdot L=1$. The proposition holds for this case.
	\end{itemize}
	It remains to show that $D^2$ cannot be positive. 
	Since the fan of $X$ contains at least 5 rays, by Lemma \ref{L^2>=LD+4}, 
	\begin{equation}\label{eq2}
		L^2\ge L\cdot D+4.
	\end{equation}
	In addition, it follows from Corollary \ref{Hodge inequality} that
	\begin{equation}\label{eq4}
		(L\cdot D)^2\ge L^2\cdot D^2.
	\end{equation}
	Combining \eqref{eq4} with \eqref{eq1} and \eqref{eq2} yields
	\[(L\cdot D)^2\ge (L\cdot D-1)(L\cdot D+4)=(L\cdot D)^2+3L\cdot D-4.\]
	This implies $L\cdot D\le 1$. The only possibility is $L\cdot D=1$. Then by \eqref{eq4}, $D^2=L^2=1$, which is impossible since $L^2\ge L\cdot D+4=5$. Therefore, it cannot be the case that $D^2>0$.
	
	We now prove the second part of the proposition. Suppose that $K_X+L$ is not ample, so there exists $\rho\in \Sigma(1)$ such that $(K_X+L)\cdot D_{\rho}\le 0$. Let $D=D_{\rho}$. By Corollary \ref{(L+KX)D=LD-D^2-2},
	\begin{align*}
		(L+K_X)\cdot D=L\cdot D-D^2-2\le 0.
	\end{align*}
	This implies $L\cdot D\le  D^2+2$; hence,
	\begin{equation}\label{eq5}
		1\le L\cdot D\le D^2+2.
	\end{equation}
	\begin{itemize}
		\item If $D^2=-1$, then $1\le L\cdot D\le 1$, so $L\cdot D=1$.
		\item If $D^2=0$, either $D\cdot L=1$ or $D\cdot L=2$.
	\end{itemize}
	Now we consider the case that $D^2\ge 1$. Since the fan of $X$ contains at least 5 rays, by Lemma \ref{L^2>=LD+4},
	\begin{equation}\label{eq6}
		L^2\ge  L\cdot D+4.
	\end{equation}
	By Corollary \ref{Hodge inequality},
	\begin{equation}\label{eq7}
		(L\cdot D)^2\ge L^2\cdot D^2
	\end{equation}
	Since $D^2\ge 1$, then by \eqref{eq6}, $L^2\ge 5$. Thus by \eqref{eq7}, $(L\cdot D)^2\ge L^2\cdot D^2\ge 5$, so $L\cdot D>2$. It follows that $L\cdot D\ge 3$. Hence, $L\cdot D-2\ge 1$. This inequality combining with \eqref{eq5} and \eqref{eq6} yields
	\[(L\cdot D)^2\ge (L\cdot D-2)(L\cdot D+4)=(L\cdot D)^2+2L\cdot D-8.\]
	This implies $L\cdot D\le 4$. The only possibilities are $L\cdot D=3$ or $L\cdot D=4$. 
	\begin{itemize}
		\item If $D^2=1$ then $L\cdot D\le 3$ by \eqref{eq5}. Since $L\cdot D$ can only be either $3$ or $4$, $L\cdot D=3$ in this case. Furthermore, suppose that $L^2\ge 10$. If $L\cdot D=3$ and $D^2=1$ then $9=(L\cdot D)^2<  10\le L^2\cdot D^2$, a contradiction to \eqref{eq7}.
		
		\item Now assume that $D^2\ge 2$. If $L\cdot D=3$, then $L^2\ge 7$ by \eqref{eq6}, and $L^2\cdot D^2\ge 7\cdot 2=14>9=(L\cdot D)^2$, a contradiction to \eqref{eq7}. Now assume that $L\cdot D=4$. Then the polygon $P_L$ associated to $L$ has at least $5$ vertices and one of its edges has lattice length $4$ by Lemma \ref{lattice length}. Hence, $L^2\ge 9$ by Lemma \ref{Volume of polygon >=9}. It follows that $16=(L\cdot D)^2< 18\le L^2\cdot D^2$, a contradiction to \eqref{eq7}.
	\end{itemize}
	The proposition follows.
\end{proof}

\section{Some Applications}
The following corollary gives an affirmative answer for a stronger form of Fujita's conjecture (Conjecture \ref{Fujita conjecture}) in case of smooth complete toric surfaces. Note that for $n$-dimensional toric varieties, the Fujita's conjecture is in fact a corollary of \cite[Corollary 0.2]{Fujino2003} and \cite[Theorem 1]{Payne2006}.
\begin{corollary}[{\cite{Fujino2003, Payne2006}}]
	Let $X$ be a smooth complete surface not isomorphic to $\PP^2$. Let $L$ be an ample line bundle on $X$ such that $L\cdot C\ge 2$ for all toric invariant curve $C\subset X$. Then $\shO_X(K_X+L)$ is globally generated. If $L^2\ge 10$  and $L\cdot C\ge 3$ for all toric invariant curve $C\subset X$, then $\shO_X(K_X+L)$ is very ample.
\end{corollary}
\begin{proof}
	Suppose that $\shO_X(K_X+L)$ is not globally generated. By Proposition \ref{My Reider variant}, there exists a toric invariant curve $C$ such that $L\cdot C=0$ or $L\cdot C=1$, a contradiction.
\end{proof}

As a corollary, we have a stronger form of \cite[Corollary 2.7]{Lazarsfeld1994} for smooth toric surfaces as follows.
\begin{corollary}
	If $A$ is an ample line bundle on a smooth complete toric surface $X$ not isomorphic to $\PP^2$, then $|K_X+2A|$ is nef, and $|K_X+4A|$ is very ample.
\end{corollary}
\begin{proof}
	Take $L=2A$, then for any toric invariant curve $C\subset X$, $L\cdot C=2A\cdot C\ge 2$. By Proposition \ref{My Reider variant}, $|K_X+2A|$ is nef. Similarly, take $L'=4A$, then $(L')^2=16A^2>10$, and $L\cdot C=4A\cdot C\ge 4$. By Proposition \ref{My Reider variant}, $|K_X+4A|$ is very ample.
\end{proof}

\begin{remark}
	It would be interesting to see if we can apply the classification in Proposition \ref{My Reider variant} to the study of Iskovskikh-Shokurov conjecture \cite{Iskovskikh2005} for conic bundles over smooth toric surfaces.
\end{remark}

%
%
\providecommand{\bysame}{\leavevmode\hbox to3em{\hrulefill}\thinspace}
\providecommand{\MR}{\relax\ifhmode\unskip\space\fi MR }
\providecommand{\MRhref}[2]{%
	\href{http://www.ams.org/mathscinet-getitem?mr=#1}{#2}
}
\providecommand{\href}[2]{#2}

\end{document}